\documentclass[reqno,a4paper]{amsart} \usepackage{amsmath} \usepackage{amssymb,amsfonts} \usepackage{amsthm} \usepackage{enumerate} \usepackage[mathscr]{eucal} \usepackage{eqlist}
\usepackage{cite}

  \theoremstyle{plain} \newtheorem{theorem}{Theorem}[section] \newtheorem{prop}[theorem]{Proposition}
\newtheorem{lemma}{Lemma}[section] \newtheorem{corol}{Corollary}[theorem]  \theoremstyle{definition} 
  \theoremstyle{remark} 
\numberwithin{equation}{section}     
   \theoremstyle{theorem} \newtheorem{theoremA}{Theorem}

\begin{document}

\title[The Shuffle Variant of a Diop. eq. of Miyazaki and Togb\'{e}] {The Shuffle Variant of a Diophantine equation of Miyazaki and Togb\'{e}}

\author{El\.{I}f K{\i}z{\i}ldere, G\"{o}khan Soydan, Qing Han and Pingzhi Yuan}

\address{{\bf Elif K{\i}z{\i}ldere}\\ Department of Mathematics \\ Bursa Uluda\u{g} University\\
 16059 Bursa, Turkey}
\email{elfkzldre@gmail.com}

\address{{\bf G\"{o}khan Soydan} \\ 	Department of Mathematics \\ 	Bursa Uluda\u{g} University\\ 	16059 Bursa, Turkey} \email{gsoydan@uludag.edu.tr }

\address{{\bf Qing Han} \\ Faculty of Common Courses\\
 South China Business College of Guangdong University of Foreign Studies Guangzhou, 510545, China}
 \email{46620467@qq.com}

\address{{\bf Pingzhi Yuan} \\ 	School of Mathematics\\ 	South China Normal University\\ 	Guangzhou 510631, China} \email{mcsypz@mail.sysu.edu.cn}

\newcommand{\acr}{\newline\indent}

\thanks{}

\subjclass[2010]{11D61, 11J86} \keywords{Exponential Diophantine equation, Baker's method}

\begin{abstract}
In 2012, T. Miyazaki and A. Togb\'{e} gave all of the solutions of the Diophantine equations $(2am-1)^x+(2m)^y=(2am+1)^z$ and $b^x+2^y=(b+2)^z$ in positive integers $x,y,z,$ $a>1$ and $b\ge 5$ odd. In this paper, we propose a similar problem (which we call the shuffle variant of a Diophantine equation of Miyazaki and Togb\'{e}). Here we first prove that the Diophantine equation $(2am+1)^x+(2m)^y=(2am-1)^z$ has only the solutions $(a, m, x, y, z)=(2, 1, 2, 1, 3)$ and $(2,1,1,2,2)$ in positive integers $a>1,m,x,y,z$. Then using this result, we show that the Diophantine equation $b^x+2^y=(b-2)^z$ has only the solutions $(b,x, y, z)=(5, 2, 1, 3)$ and $(5,1,2,2)$ in positive integers $x,y,z$ and $b$ odd.
\end{abstract}

\maketitle

\section{Introduction}\label{sec:1}

Denote the sets of all integers and positive integers by $\mathbb{Z}$ and $\mathbb{N}$, respectively. Suppose that $A, B, C$ are pairwise coprime positive integers. The exponential Diophantine equation \begin{equation}\label{1.1} 	A^x+B^y=C^z, \ \ x, y, z \in \mathbb{N} \end{equation} was studied for given $A,B,C$ by many authors. So, this equation has a rich history. In 1933,
the first work was recorded by Mahler, \cite{Ma}. He proved the finiteness of the solutions of equation \eqref{1.1} under the assumption that $A,B,C>1$. His method is a $p$-adic analogue of that given by Thue-Siegel, so it is ineffective in the sense that it gives no indication on the number of possible solutions. Seven years later, an effective result for solutions of \eqref{1.1} was given by Gel'fond what will become known as Baker's method, based on lower bounds for linear forms in the logarithms of algebraic numbers, \cite{Ge}. Such an information has been obtained in particular instances.
So, in 1956, Sierpi\'{n}ski proved that $(x,y,z)=(2,2,2)$ is the only positive integral solution of the equation $3^x+4^y=5^z$, \cite{Si}. The same year, Je\'{s}manowicz conjectured that if $A, B, C$ are Pythagorean numbers, i.e. positive integers satisfying $A^2+B^2=C^2$, then the Diophantine equation $A^x+B^y=C^z$ has only the positive integral solution
$(x,y,z)=(2,2,2)$, \cite{J}. This conjecture is still open despite the efforts of many authors. Between the years 1958 and 1976, the complete solutions of equation \eqref{1.1} where $A, B, C$ are distinct primes $\leq 17$ were determined by some authors (see \cite{Ha}, \cite{Na} and \cite{U}). Other conjectures related to equation \eqref{1.1} were set and discussed. One is the
extension of Je\'{s}manowicz' conjecture due to Terai. In fact, Terai conjectured that if $A, B, C, P, Q, R \in \mathbb{N}$ are fixed and $A^P+B^Q=C^R$ where $P, Q, R \geq 2$ and $\gcd(A,B)=1$, then the Diophantine equation \eqref{1.1} has only the solution $(x,y,z)=(P,Q,R)$ except for a handful of triples $(A,B,C)$, (see \cite{Cao, Le, Mi1, Mi2} and \cite{T1, T2, T3, T4, T5}). This conjecture has been proved to be true in many special cases. However, it is still unsolved in its full generality. Recently, a survey paper on the conjectures of Je\'{s}manowicz
and Terai has been written by Soydan, Demirci, Cang\"{u}l and Togb\'{e}, (see \cite{SCDT}).

For the special case where $A \equiv -1 \pmod{B}$ and $C=A+2$, the equation takes the form \begin{equation}\label{1} (tB-1)^x+B^y=(tB+1)^z, \end{equation} where $t$ is a positive integer. Clearly, it suffices to consider the case where $B$ is even. We see that the equation \eqref{1} has the following solutions: \begin{equation*} \label{allsol} (x,y,z)=\begin{cases} \ (i,1,1);i \ge
1, \ (j,3,2);j \ge 1 & \text{if $B=2$ and $t=1$,}\\ \ (2,k\!+\!1, 2) & \text{if $t=B^k/4$ with $k \ge 1$,}\\ \ (1,1,1) & \text{if $B=2$,}\\ \ (1, 13, 2) & \text{if $B=2$ and $t=45$.}
\end{cases} \end{equation*} These solutions will be referred to as {\it trivial solutions}.

In 2012, Miyazaki and Togb\'{e} proved that equation \eqref{1} has no non-trivial solutions when $t$ is odd, \cite{MiTo}. In 2016, Miyazaki, Togb\'{e} and Yuan gave the following result in \cite{MiToYu}:
\begin{theoremA}\label{thm:A1} Equation \eqref{1} has no non-trivial solutions.
\end{theoremA}

Using Theorem \ref{thm:A1}, they also proved the following result:
\begin{theoremA}\label{thm:A2} Suppose that $a>1$ is an odd positive integer. Then the Diophantine equation
\begin{equation*} a^x+2^y=(a+2)^z
\end{equation*} has only the positive solution $(x,y,z)=(1,1,1)$, whenever neither
\begin{equation*} a=2^{k-1}-1
\end{equation*}
with an integer $k\ge 3$
nor $a=89$. If $a=2^{k-1}-1$ or $a=89$, then the additional solutions are given by $(2,k+1,2)$, $(1,13,2)$, respectively. \end{theoremA}

More recently, Fu, He, Yang and Zhu in \cite{FHYZ} considered the Diophantine equation \begin{equation}\label{1.3} (n+2)^x+(n+1)^y=n^z  \ \ n, x, y, z \in \mathbb{N}. \end{equation} They
obtained the following result (For variations of equation \eqref{1.3}, we refer the reader to \cite{BKSY} and \cite{HT}): \begin{theoremA}\label{thm:A3} Equation \eqref{1.3} has only one
positive integer solution $(n,x,y,z)$ $=(3,1,1,2)$. \end{theoremA} In this work, we propose analogs of Theorem \ref{thm:A1} and Theorem \ref{thm:A2} which we call
the shuffle variants of Diophantine equations of Miyazaki and Togb\'{e}, \cite{MiTo}. So, here we are interested in the following Diophantine equation \begin{equation*}
(an+a+1)^x+(n+1)^y=(an+a-1)^z;  \ \ n, x, y, z \in \mathbb{N}. \end{equation*} where $a$ is a fixed positive integer. According to Theorem \ref{thm:A3}, since the case $a=1$ was already solved completely by Fu, He, Yang and Zhu, \cite{FHYZ}, it is clear that the above equation has a solution only if $n$ is odd. So, we only need to consider the following Diophantine
equation \begin{equation*} (2am+1)^x+(2m)^y=(2am-1)^z; \ \ m, x, y, z \in \mathbb{N}. \end{equation*}
Consider the above equation. Here, we extend Theorem \ref{thm:A3} by proving the following result which is a shuffle variant of equation \eqref{1} in Theorem 1.2 of \cite{MiTo}.
\begin{theorem}[Main theorem]\label{theo.1.1}
The Diophantine equation \begin{equation}\label{eq1.4} (2am+1)^x+(2m)^y=(2am-1)^z
\end{equation} has only the solutions $(a, m, x, y, z)=(2, 1, 2, 1, 3)$ and $(2,1,1,2,2)$ in positive integers $a>1,$ $m,x,y,z$.
\end{theorem}

Furthermore, using Theorem \ref{theo.1.1}, we prove the following result which is the shuffle variant of the equation \eqref{1.3} in Theorem 1.2 of \cite{MiTo}.
\begin{corol}\label{Cor.1.1.1}
The Diophantine equation \begin{equation}\label{eq1.5}
b^x+2^y=(b-2)^z
\end{equation}
has only the solutions $(b,x, y, z)=(5, 2, 1, 3)$ and $(5,1,2,2)$ in positive integers $x,y,z$ and $b$ odd.
\end{corol}

\section{A key lemma}
The following lemma  and its proof are almost the same as the key lemma in Miyazaki, Togb\'{e} and Yuan, \cite{MiToYu}. For the convenience of the reader, we present the proof here.

For a prime number $p$ and a non-zero integer $A$, we denote by $v_{p}(A)$ the exponent of $p$ in the prime factorization of $A$.

\begin{lemma}\label{key} Let $(a,m,x,y,z)$ be a solution of equation \eqref{eq1.4} with $m>1$. \begin{itemize} \item[\rm (i)] If $x$ is odd, then $y=\frac{v_{2}(a)}{v_{2}(m)+1}+1$.
\item[\rm (ii)] $2am \ge (2m)^y(x+z)^{-y}$. \end{itemize} \end{lemma}

\begin{proof} Notice that \begin{align*} &(2am-1)^z \equiv (-1)^{z}+(-1)^{z-1}2amz \pmod{4a^2m^2},\\ &(2am+1)^x \equiv 1+2amx \pmod{4a^2m^2}. \end{align*} Taking equation
\eqref{eq1.4} modulo $4a^2m^2$, one sees that $$ 1+2amx +(2m)^y \equiv (-1)^{z}+(-1)^{z-1}2amz \pmod{4a^2m^2}. $$ Considering this congruence modulo $2m$, one has $(-1)^z \equiv 1 \pmod{2m}$. Then we see that $z$ is even
as $m>1$. Hence, we get $$ a(x+z) \equiv -(2m)^{y-1} \pmod{2ma^2}. $$ Using the above congruence, if a prime factor $p$ of $2m$ satisfies $$ v_{p}(x+z)<v_{p}(a)+v_{p}(2m), $$ then we
obtain \begin{equation}\label{imp} v_{p}(a)=(y-1)\,v_{p}(2m)-v_{p}(x+z). \end{equation}

(i) Putting $p=2$ in \eqref{imp} completes the proof of this case.

(ii) Using \eqref{imp}, one obtains \begin{align*} a &\ge \prod_{\substack{p\,{\rm prime}, \ p \mid 2m, \\ v_{p}(x+z)<v_{p}(a)+v_{p}(2m)}} p^{v_{p}(a)} \\ &= \prod_{\substack{p\,{\rm prime}, \ p
\mid 2m, \\ v_{p}(x+z)<v_{p}(a)+v_{p}(2m)}} p^{\,(y-1)\,v_{p}(2m)-v_{p}(x+z)} \\ &=(2m)^{y-1} S^{-(y-1)} \prod_{\substack{p\,{\rm prime}, \ p \mid 2m, \\v_{p}(x+z)<v_{p}(a)+v_{p}(2m)}}
p^{-v_{p}(x+z)}, \end{align*} where $$ S=\prod_{\substack{p\,{\rm prime}, \ p \mid 2m, \\ v_{p}(x+z) \ge v_{p}(a)+v_{p}(2m)}} p^{v_{p}(2m)}. $$ Moreover, since $S \le x+z$, one gets
$$ S^{\,y-1} \prod_{\substack{p\,{\rm prime}, \ p \mid 2m, \\ v_{p}(x+z)<v_{p}(a)+v_{p}(2m)}} p^{v_{p}(x+z)} \le (x+z)^y. $$ 
The required inequality follows from all these.
\end{proof}

\section{Auxiliary results}

For an algebraic number $\alpha$ of degree $d$ over $\mathbb{Q}$, we define the absolute logarithmic height of $\alpha$ by the following formula: $$ {\rm h}(\alpha)= \dfrac{1}{d} \left( \log
\vert a_{0} \vert + \sum\limits_{i=1}^{d} \log \max \bigr\{ 1, \vert \alpha^{(i)}\vert \bigr\} \right), $$ where $a_{0}$ is the leading coefficient of the minimal polynomial of $\alpha$ over
$\mathbb{Z}$, and $\alpha^{(1)}, \alpha^{(2)}, \,...\,, \alpha^{(d)}$ are the conjugates of $\alpha$ in the field of complex numbers.

Let $\alpha_{1}$ and $\alpha_{2}$ be real algebraic numbers with $|\alpha_{1}|\ge 1$ and $|\alpha_{2}|\ge 1$. Consider the linear form in two logarithms $$
\varLambda=\beta_{2}\log\alpha_{2}-\beta_{1}\log \alpha_{1}, $$ where $\beta_{1}$ and $\beta_{2}$ are positive integers.

We rely on the following result due to Laurent, \cite{La}:

\begin{prop}\cite[Corollary 2] {La} \label{La} Suppose that $\alpha_1>1$, $\alpha_2>1$ are rational numbers which are multiplicatively independent. Then we have $$\log\vert\varLambda\vert
\ge -25.2 \cdot h(\alpha_1)\, h(\alpha_2) \,\max\{\log{\beta'}+0.38,10\}^2$$ where \begin{equation*} \beta'=\frac{\beta_1}{h(\alpha_2)}+\frac{\beta_2}{h(\alpha_1)}. \end{equation*}
\end{prop}

\begin{prop}\cite[Table 1]{M15} \label{po4} Let $k>3$ be a positive integer. Then the Diophantine equation $$x^{2}+y^k=z^{4}; \quad x, y, z\in\mathbb{N}$$ has no solutions with $\gcd(x,y)=1$.\end{prop}

We also  need the following result due to Le \cite{Le1}. \begin{prop}[\cite{Le1}] \label{prop.1} The solutions of the equation \begin{equation*} U^2+2^k=V^l; \quad U, V, k, l\in\mathbb{N},\,\,
gcd(U,V)=1,\,\, l\ge 3 \end{equation*} are given by $(U,V,k,l)=(5,3,1,3), (7,3,5,4), (11,5,2,3).$
\end{prop}

\section{Proof of Theorem \ref{theo.1.1}} 
In this section, we prove Theorem \ref{theo.1.1}. We separate the cases $y=1$ and $y>1$. It is easy to see that $z>x$.
\subsection{The case $y=1$}
Here, we consider equation \eqref{eq1.4} where $m>1$ and $m=1$, respectively.

\subsubsection{The case $m>1$}
Reducing equation \eqref{eq1.4} modulo $2m$, we get  $1+0 \equiv (-1)^z \pmod {2m}$. Hence, $2\mid z$. Considering \eqref{eq1.4} modulo $2am$, we obtain $1+2m\equiv 1 \pmod{2am}$, a contradiction with $a>1$. 
\subsubsection{The case $m=1$}\label{**} When $m=1$, equation \eqref{eq1.4} becomes \begin{equation}\label{***} (2a+1)^x+2=(2a-1)^z. \end{equation} If $2\mid z$, reducing equation \eqref{***} modulo $2a$, we get $a=1$, which contradicts the assumption that $a>1$. If $2\nmid z$, considering equation \eqref{***} modulo $2a$, one gets $2a\mid 4$. Since $a>1$, we have
$a=2$. When $a=2$, equation \eqref{***} becomes $5^x+2=3^z$ which has only the solution $(x,z)=(2,3)$ by \cite[Theorem 3]{Na}. Since this is the desired solution of equation \eqref{eq1.4}, the proof of the case $y=1$ is completed.

\subsection{The case $y>1$}

\subsubsection{The case $2\mid x$ and $2\mid z$} Suppose that $x=2X$ and $z=2Z$. From equation \eqref{eq1.4}, we introduce two even positive integers $P$ and $Q$ as follows: \begin{equation}\label{eq2.1}
(2m)^y = PQ,
\end{equation} where
\begin{equation}\label{eq2.2}
P=(2am-1)^Z+(2am+1)^X, \ Q=(2am-1)^Z-(2am+1)^X.
\end{equation}
Since $Q$ is congruent to $0$ or $-2$ modulo $2am$, one has $PQ>(2am-1)^Z\cdot(2am-2)$. Since $Z=z/2>x/2=X \ge 1$, one obtains $(2am-1)^Z\cdot(2am-2)\ge (2am-1)^2\cdot(2am-2)$. Hence  $(2m)^y>(2am-1)^2\cdot(2am-2)> (2m)^3$ because $a>1$. So, $y>3$ and by Proposition \ref{po4}, we obtain $2\nmid Z$, whence  $Z\ge3$.

By equation \eqref{eq2.2}, we see that $\gcd (P,Q)=2$ and $P+Q\equiv 2 \pmod 4$, so $P/2$ and $Q/2$ are integers of different parities. Further $Q/2$ is coprime to $m$. Since $2^{y-2}m^y=(P/2)(Q/2)$ and $Q/2$ is coprime to $m$, it follows that only $Q/2=1$ or $Q/2=2^{y-2}$ are possible.
\begin{itemize}
\item The case $2\mid am.$
\end{itemize}

Consider equation \eqref{eq1.4} for the case $2\mid am$. When $2\mid am$, $Q/2\equiv -1 \pmod{am}$, only the case $Q=2$ is possible. Hence, one has \begin{equation}\label{xxy}
(2am-1)^{Z}-(2am+1)^{X}=2,\quad (2am-1)^{Z}+(2am+1)^{X}=2^{y-1}m^{y}. \end{equation}
Now consider equation
\begin{equation}\label{eq23}
(2am-1)^Z-(2am+1)^X=2.
\end{equation} 
Taking \eqref{eq23} modulo $2am$, one finds that $2am$ divides 2 or 4, so that $a=2$, $m=1$. Hence equation \eqref{eq23} becomes
\begin{equation}\label{Pil}
3^Z-5^X=2.
\end{equation}
By \cite[Theorem 6]{SS}, one  sees that equation \eqref{Pil} has only the solution $(X,Z)=(2,3),$  namely, equation \eqref{eq23} has only the solution $(a,m,X,Z)=(2,1,2,3)$. However, this contradicts the second equation in \eqref{xxy}.
\newpage
\begin{itemize}
\item The case $2\nmid am.$
\end{itemize}

Consider equation \eqref{eq1.4} for the case $2\nmid am$. We first deal with the case $2\nmid X$. By the former case, we know that equation \eqref{eq23} has no solutions. Thus, one gets $$(2am-1)^Z+(2am+1)^X=2m^y,\,\, 2\nmid m.$$ Reducing the above equation modulo 4, we obtain $$2am(X+Z)\equiv2\pmod{4},$$ which is impossible since $2\nmid XZ$.

Now we consider the case $2\mid X$. Then we have $$Q=(2am-1)^Z-(2am+1)^X=2 \ \ \text{or} \ \ 2^{y-1}.$$ Applying Proposition \ref{prop.1}, we see that the above equation has no solutions with $2\mid X$.

\subsubsection{The case $2\nmid\gcd(x, z)$}
In this subsection, we consider the case $2\nmid\gcd(x, z)$. Moreover, we consider the cases $m = 1$ and $m>1$ separately. \begin{itemize} \item The case $m=1.$ \end{itemize}

Consider the following Diophantine equation \begin{equation}\label{eq.4.5} (2a+1)^x+2^y=(2a-1)^z; \,\, y>1, \,\, 2\nmid\gcd(x,z). \end{equation} If $2\mid x$, then by Proposition \ref{prop.1}, we get $a=2$, $x=2$ and $y=1$. Hence we suppose that $2\nmid x$. If $2\nmid a$, then $2a+1\equiv3\pmod{4}$ and $2a-1\equiv 1\pmod{4}$, we get a contradiction reducing equation \eqref{eq.4.5} modulo $4$.

If $2\mid a$, then reducing equation \eqref{eq.4.5} modulo $4$, we get $2\mid z$. Reducing modulo $a$, we get $2^y\equiv0\pmod{a}$, so $a=2^u$, $u\ge1$. Further, we have $$2a(x+z)\equiv-2^y\pmod{2a^2}.$$ It follows that $u=y-1$. Equation \eqref{eq.4.5} becomes $$(2^y+1)^x+2^y=(2^y-1)^z.$$ By Theorem \ref{thm:A3}, we get the solution $(x,y,z)=(1,2,2)$. Since this is the desired solution of equation \eqref{eq1.4}, the proof of the case $m=1$ is completed.

\begin{itemize} \item The case $m>1.$ \end{itemize}

Since $m>1$, considering equation \eqref{eq1.4} modulo $2m$, one gets $1\equiv(-1)^z\pmod{2m}$, so $z$ is even. Hence, since we are in the case $2\nmid\gcd(x, z)$, we conclude that $x$ must be odd.

Now we will observe that this leads to a contradiction. By Lemma \ref{key}\,(i), we have \begin{equation}\label{2-t} y=\frac{v_{2}(a)}{v_{2}(2m)}+1 \le v_{2}(a)+1 \le \frac{\,\log a\,}{\log
2}+1. \end{equation} In what follows, we put $$ A=2am+1, \quad C=2am-1 \ (=A-2).$$ We separately consider the cases $A^x \le (2m)^{y}$ and $A^x>(2m)^{y}$.

\quad

\textbf{(i) The case $A^x \le (2m)^{y}$}

\quad

Now we suppose that $A^x \le (2m)^{y}$. Since $A^x<C^z\le2(2m)^{y}$ and $2m \ge 4$, by \eqref{2-t}, we have \begin{equation*} z\le
 \frac{\log 2+\bigr(\frac{\log a}{\log 2}+1\bigr)\log 2m }{\log (2am-1)}<1.5\log (2m). \end{equation*}

Recall that $z> x$. If $z=2$, then $x=1$ and equation \eqref{eq1.4} becomes $2am+1+(2m)^y=4a^2m^2-4am+1$. Hence, one gets
\begin{equation}\label{21}
(2m)^{y-1}=a(2am-3).
\end{equation}
Note that $2am-3>3$ since $a>1$ and $m>1$. Also, $3\mid m$ as $\gcd(2m,2am-3)>1$. Dividing the above equation by 3 leads to
\begin{equation*}
2^{y-1}3^{y-2}(m/3)^{y-1}=a(2a\cdot m/3-1).
\end{equation*}
It is easy to see that $2^{y-1}(m/3)^{y-1}\mid a,$ so that
\begin{equation*}
3^{y-2}=\dfrac{a}{2^{y-1}(m/3)^{y-1}}\cdot (2a\cdot m/3-1).
\end{equation*}
Since $\gcd(\dfrac{a}{2^{y-1}(m/3)^{y-1}},2a \cdot m/3-1)=1$, and $2a \cdot m/3-1>3/3=1,$ one has
\begin{equation*}
\dfrac{a}{2^{y-1}(m/3)^{y-1}}=1,\,\,2a\cdot m/3-1=3^{y-2}.
\end{equation*}
This implies that $(2m/3)^y-1=3^{y-2}$. It is easy  to see that this equation does not hold. Thus, one has \begin{equation}\label{xz} 4\le z \le
\lfloor 1.5 \log (2m)\rfloor \end{equation} which implies that $m \ge 8$. Moreover, equation \eqref{eq1.4} easily yields $(2m)^y>a(2m)^z$ and then one gets \begin{equation}\label{yz}
y>z\ge4. \end{equation} Since Lemma \ref{key}\,(ii) and the inequality \eqref{xz} imply that \begin{align*} \left( \frac{m}{\lfloor 1.5 \log (2m)\rfloor} \right)^{y}<C+1 \le
(2m)^{y/z}\cdot2^{1/z}+1<1.5 (2m)^{y/z}, \end{align*} we have \begin{equation}\label{x} z<\frac{\log (2m)}{\log (m) -\log \bigr( 1.5^{1/y}\lfloor 1.5 \log (2m)\rfloor\bigr)}. \end{equation}
In view of \eqref{xz}, \eqref{yz} and \eqref{x}, we have $$ m=8, \quad z =4, \quad x \in \{ 1,3\}. $$

If $z=4$, $x=1$ and $m=8$, then equation \eqref{eq1.4} becomes
\begin{equation}\label{22}
a(4096a^3-1024a^2+96a-5)=2^{4y-4}.
\end{equation}
Hence one obtains $a=2^{4y-4}$ and $4096a^3-1024a^2+96a-5=1.$ Taking the second inequality modulo 4 yields a contradiction.

If $z=4$, $x=3$ and $m=8$, then equation \eqref{eq1.4} becomes \begin{equation}\label{23}
a(4096a^3-1280a^2+48a-7)=2^{4y-4}.
\end{equation}
Similarly to the former case, we get that $4096a^3-1280a^2+48a-7=1$. Taking this equality modulo 16 yields a contradiction.

\quad

\textbf{(ii) The case $A^x>(2m)^{y}$}

\quad

Here, we suppose that $A^x>(2m)^{y}$. Then one gets $(2am+1)^x>\frac{1}{2}(2am-1)^z$ or $$2>\left(\frac{2am-1}{2am+1}\right)^x\cdot(2am-1)^{z-x}.$$
Since $z-x\ge1$, we have $$e<(2am-1)^{z-x}/2<\left(\frac{2am+1}{2am-1}\right)^x<e^{\frac{2x}{2am-1}},$$ which implies that $x>\frac{2am-1}{2}$. Thus, $$x\ge am.$$ Recall that $x$ is odd. By \eqref{2-t}, we have $y\le v_2(a)+1$. Thus, $$y\log(2m)\le(v_2(a)+1)\log(2m)\le a\log(2m)<\frac{am}{2}\log(2m)\le\frac{x}{2}\log A,$$ and we have $(2m)^y<A^{x/2}$. Thus,
$C^z=A^x+(2m)^y<A^x+A^{x/2}$, i.e., $$C^{z-x}<\left(\frac{2am+1}{2am-1}\right)^x\left(1+A^{-x/2}\right).$$

All solutions with $z\in\{2,4\}$ have been found in the case $A^x<(2m)^{y}$.

Since $C=2am-1\ge 2\cdot 2\cdot 2-1=7$ and $x\ge5$, we have $$(z-x)\log
C<x\log\left(1+\frac{2}{C}\right)+\log\left(1+A^{-x/2}\right)<x\left(\frac{2}{C}-\frac{4}{3C^2}\right)+A^{-x/2}<\frac{2x}{C}.$$ Therefore, we get \begin{equation}\label{x-z} x> \frac{C\log
(C)}{2}\,(z-x). \end{equation}

Put $$ \varLambda = z \log C - x \log A \ (>0). $$ Since $\varLambda<\exp(\varLambda)-1=(2m)^y/A^x<A^{-x/2}$, one has $$ \log \varLambda <- (x/2) \log A. $$ On the other hand, to find a
lower estimate of $\log \varLambda$, we apply Proposition \ref{La} with $(\alpha_1,\alpha_2)=(C,A)$ and $(\beta_1,\beta_2)=(z,x)$. Hence, we obtain $$ \log \varLambda > -25.2 \,(\log A)(\log
C) \bigr(\! \max  \{ \log \beta'+0.38,\,10\} \bigr)^{2}, $$ where $\beta'=\frac{x}{\log C}+\frac{z}{\log A}$. Since $2A^x>C^z$, the inequality $\beta'<\frac{2x}{\log C}+1$ gives $$ s <
50.4\, \bigr(\! \max\{\log (2s+1) +0.38,\,10 \} \bigr)^{2}, $$ where $s=x/\log C$. This implies that $s<5040$. Therefore, by \eqref{x-z}, one gets $A<10082$. From  here, using \eqref{2-t}
and $m \ge 2$, one gets $y \le 6$ when $m$ is even, and $y\le11$ when $m$ is odd.

We claim that
\begin{equation} \label{A}
A<5044.
\end{equation}
Assume for a contradiction that \eqref{A} does not hold, so $A\ge 5044$. Then,  \eqref{x-z} gives $x\ge 21493$. Since $y \le 11$, we see that $A^x>(2m)^{1953y}$
clearly holds. Since $\varLambda<\exp(\varLambda)-1=(2m)^y/A^x<A^{-1952x/1953}$, one has $$ \log \varLambda <- (1952x/1953) \log A. $$ Using Proposition \ref{La} as above, we get $$ s <
25.2\times\frac{1953}{1952}\, \bigr(\! \max\{\log (2s+1) +0.38,\,10 \} \bigr)^{2}$$ where $s=x/\log C$. This implies that $s<2522$ and $A<5044$ which is the desired contradiction. Hence, \eqref{A} holds. Note that $y \le 10$ by (\ref{2-t}) and the same argument as before.

Finally, we show that \begin{equation} \label{x2} x<2522 \,\log (A-2) \quad \mbox{or} \,\, \ \ x \le 1300y. \end{equation} If $x>1300y$, then $A^x>(2m)^{1300y}$, which implies that $s<2522$, that is, \eqref{x2} holds.

We can check that equation \eqref{eq1.4} does not hold for any $(a,m,x,y,z)$ satisfying all \eqref{2-t}, \eqref{x-z}, \eqref{A}, \eqref{x2} and $y\le10$. For this, a program was written in
PARI/GP, \cite{PARI}, and it took about 1 hour to run the program for each value of $y$ with the restriction $a\mid (2m)^{y-1}.$ This completes the proof of Theorem \ref{theo.1.1}.

\section*{Acknowledgments} We would like to thank the referees for carefully reading our paper and for giving such constructive comments which substantially helped improving the quality of the paper. The first and second authors would like to thank Professor Takafumi Miyazaki for drawing their attention to this problem. The second author would like to thank Dr. Paul Voutier for giving us useful ideas for speeding up our PARI-GP program. This work was started when the first and second authors participated to the conference \textquotedblleft Diophantine $m$-tuples and related problems-II'' on 15--17th October, 2018 in Purdue University Northwest, Westville/Hammond, in USA. They would like to thank Professors Bir Kafle and Alain Togb\'e for this nice organization and their kind hospitality and were supported by T\"{U}B\.{I}TAK (the Scientific and Technological Research Council of Turkey) under Project No: 117F287. The fourth author was supported by NSF of China (No. 11671153) and NSF of Guangdong Province (No. 2016A030313850).

\end{document}